\documentclass[a4paper]{article}
\usepackage{graphicx}
\usepackage{url}
\usepackage[margin=3cm]{geometry}
\usepackage{amsmath,amssymb}
\usepackage{amsthm}
\usepackage[hidelinks]{hyperref}
\usepackage{xcolor}
\usepackage{tikz}
\newtheorem{theorem}{Theorem}
\newtheorem{lemma}[theorem]{Lemma}

\newtheorem{claim}[theorem]{Claim}
\usepackage{thm-restate}
\usepackage{paralist}

\newcommand{\claimproofstart}[1][Proof]{\begin{proof}[#1]
\renewcommand{\qedsymbol}{$\boxdot$}}
\newcommand\claimproofend{
\end{proof}
\renewcommand{\qedsymbol}{$\square$}}
\newcommand{\su}{\subseteq}

\title{Poset dimension and maximum comparability degree}
\author{
Carla Groenland\thanks{Delft Institute of Applied Mathematics, TU Delft, the Netherlands. Email: 
{\tt c.e.groenland@tudelft.nl}. Supported by the Dutch Research Council (NWO, VI.Veni.232.073).}\and Richard Montgomery\thanks{Mathematics Institute, University of Warwick, Coventry, United Kingdom. Email: {\tt
richard.montgomery@warwick.ac.uk}. Supported by the European Research Council (ERC) under the European Union Horizon 2020 research and innovation programme (grant agreement No.\ 947978).}
\and
Rajko Nenadov\thanks{School of Mathematics and Statistics, University of Canterbury, New Zealand. Email: {\tt rajko.nenadov@canterbury.ac.nz}. Supported by the New Zealand Marsden Fund 23-UOA-117.}
\and Lisa Sauermann\thanks{Institute for Applied Mathematics, University of Bonn, Germany. Email: {\tt lsauerma@uni-bonn.de}. Supported by the
DFG Heisenberg Program (project no.\ 515036630).}}

\date{\today}

\begin{document}
\maketitle

\begin{abstract}
    In 1986, F{\"u}redi and Kahn showed that the dimension $\dim(P)$ of any finite poset $P$ satisfies $\dim(P) = O(d \log^2 d)$, where $d$ is the maximum degree of the comparability graph of $P$. Scott and Wood more recently improved this bound to one of the form $d \log^{1+o(1)} d$. We show that  $\dim(P) = O(d \log d)$, thus confirming that the corresponding lower bound of Erd\H{o}s, Kierstead, and Trotter is tight up to the implicit constant.
\end{abstract}

\section{Introduction}

The dimension of a poset $P$ is the minimum possible size of a  collection $\{L_1,\dots,L_d\}$ of linear extensions of $P$ such that, for every pair of distinct elements $x$ and $y$ with $x\not\prec_P y$, we have $y\prec_{L_i} x$ for one of these linear extensions $L_i$ (see also Section \ref{sec:preliminaries} for further background).

Here, we study an old problem about bounding the dimension $\dim(P)$ of a poset $P$ in terms of the maximum degree $d$ of the comparability graph of $P$. The first result in this direction was by R\"odl and Trotter, who showed that any poset $P$ with maximum comparability degree $d$ satisfies $\dim(P) \le 2d^2+2$ (see~\cite{furedi1986dimensions} and~\cite{trotter1996graphs}), before F{\"u}redi and Kahn~\cite{furedi1986dimensions} improved this in 1986 to $\dim(P) = O(d \log^2 d)$. In the 1990s, 
Erd\H{o}s, Kierstead, and Trotter~\cite{erdos1991dimension} showed that there exist posets $P$ with maximum comparability degree $d$ such that $\dim(P) = \Omega(d \log d)$. Their example is a poset of height 2 with an otherwise randomly chosen comparability graph (see~\cite{erdos1991dimension} for details).

More recently, Scott and Wood~\cite{scott2020better} improved the long-standing F\"uredi-Kahn bound by proving that any poset with maximum comparability degree $d$ satisfies
$$
\dim(P) \le d\log d\cdot e^{O(\sqrt{\log\log d})}.
$$ 
Here, we improve this, showing that the lower bound of Erd\H{o}s, Kierstead, and Trotter~\cite{erdos1991dimension} is tight up to a constant factor, as follows.

\begin{theorem}\label{thm:main}
Let $d\ge 5$ be an integer. Every finite poset $P$ whose comparability graph has maximum degree at most $d$ satisfies $\dim(P)\le 50 d\log d$.
\end{theorem}

Scott and Wood~\cite{scott2020better} proved their bound towards Theorem~\ref{thm:main} via the notion of boxicity. The connection between boxicity and the problem studied here was shown by Adiga, Bhowmick, and Chandran~\cite{adiga2011boxicity}, and through this Theorem~\ref{thm:main} shows that the maximum boxicity of graphs with maximum degree $d$ is $O(d\log d)$ (see~\cite{scott2020better} for more details).

After recalling some basic definitions concerning posets, in Section~\ref{sec:preliminaries} we give a reduction of Theorem~\ref{thm:main} to a special case, presented as Theorem~\ref{thm:bipartite_dimension}, using an old result of Kimble (see, for example,~\cite{trotter1996graphs}).  We then outline the proof of Theorem~\ref{thm:bipartite_dimension}, which is carried out in Section~\ref{sec:mainproof}.

\section{Preliminaries and proof overview}\label{sec:preliminaries}

A \emph{poset} $P$ consists of a ground set $X$ together with a partial order $\preceq_P$ on $X$ (i.e., a reflexive, anti-symmetric, transitive relation on $X$). For elements $x,y\in X$ with $x \preceq_P y$ and $x\ne y$, we write $x\prec_P y$ (we omit the subscript $P$ whenever it is clear from context). We say that elements $x,y\in X$ are \emph{comparable} if $x \preceq_P y$ or $y \preceq_P x$. In this paper, all posets are assumed to be finite, meaning that their ground sets are finite.

The \emph{comparability graph} $G$ of a poset $P$ is the graph on its ground set $X$ where there is an edge between distinct vertices $x$ and $y$ if and only if $x$ and $y$ are comparable in $P$.

A \emph{linear order} $L$ on a finite set $X$ is an ordering $(x_1,\dots,x_n)$ of the elements of $X$. From such a linear order $L$, one can obtain a poset on $X$ where $x_i\prec_L x_j$ for all $i,j\in \{1,\dots,n\}$ with $i<j$. We say that the linear order $L$ is a \emph{linear extension} of a poset $P$ on the ground set $X$ if we have $x\prec_L y$ for all $x,y\in X$ with $x\prec_P y$. 

Given linear orders $L_1, \ldots, L_d$ on a set $X$, we construct the poset $P = \bigcap_{i = 1}^d L_i$ on $X$ by defining for all $x,y\in X$ that
$$
    x \prec_P y \; \Longleftrightarrow \; x \prec_{L_i} y \text{ for every } i \in \{1, \ldots, d\}.
$$
Note that then $L_1, \ldots, L_d$ are linear extensions of $P = \bigcap_{i = 1}^d L_i$.
For a given poset $P$ with ground set $X$, the \emph{dimension} $\dim(P)$ is defined to be the smallest integer $d\ge 0$ for which there exist linear orders $L_1, \ldots, L_d$ on $X$ such that $P = \bigcap_{i = 1}^d L_i$. In other words, this is the minimum possible number of linear extensions of $P$ such that, for any distinct $x,y\in X$ such that $x\not\prec y$, the element $x$ appears after $y$ in at least one of these linear extensions.

\begin{figure}[b]
\center
\begin{tikzpicture}[
    vertex/.style={
        circle,
        draw,
        fill=white,
        inner sep=0pt,
        minimum size=4pt,
        line width=0.6pt
    },
    edge/.style={line width=0.7pt},
    every node/.style={font=\large},
    scale=0.7
]

\draw (0,2.6) ellipse [x radius=3.65, y radius=0.52];
\draw (0,0)   ellipse [x radius=3.65, y radius=0.52];

\node[anchor=west] at (3.95,2.6) {$T$};
\node[anchor=west] at (3.95,0)   {$B$};

\coordinate (t1) at (-2.00,2.6);
\coordinate (t2) at (-0.25,2.6);
\coordinate (t3) at ( 1.30,2.6);
\coordinate (y)  at ( 2.25,2.6);

\coordinate (b1) at (-1.75,0);
\coordinate (b2) at (-0.10,0);
\coordinate (x)  at ( 2.25,0);

\draw[edge] (b1) -- (t1);
\draw[edge] (b1) -- (t2);
\draw[edge] (b1) -- (t3);
\draw[edge] (b2) -- (t2);
\draw[edge] (b2) -- (t3);
\draw[edge] (b2) -- (y);
\draw[edge] (x)  -- (y);

\foreach \v in {t1,t2,t3,y,b1,b2,x}
    \node[vertex] at (\v) {};

\node[right=5pt] at (y) {$v$};
\node[right=5pt] at (x) {$w$};
\node[anchor=west] at (2.3,1.30) {\small $w \prec v$};

\end{tikzpicture}
\caption{A poset with vertices in $T$ at the top and vertices in $B$ at the bottom.}
\label{fig}
\end{figure}
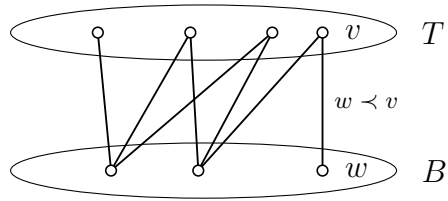

To prove Theorem \ref{thm:main}, it turns out to be sufficient to consider posets $P$ with a special structure. Namely, we can assume that the comparability graph $G$ of $P$ is bipartite with vertex classes $T$ and $B$, such that for every edge $vw$ of $G$ with $v\in T$ and $w\in B$ we have $w\prec_P v$. As in Figure \ref{fig}, we can imagine that the vertices in $T$ are at the top and the vertices in $B$ are at the bottom, such that the edges of the comparability graph $G$ correspond to the pairs $v\in T$ and $w\in B$ with $w\prec v$ (and no two elements within $T$, or within $B$, are comparable and we do not have $v\prec w$ for any $v\in T$ and $w\in B$).

We can make this assumption due to the following lemma, originally proved by Kimble in 1972 (see~\cite[Theorem~4]{TrotterMoore1976}). For convenience of the reader, we provide a proof of the lemma in the appendix.

\begin{restatable}{lemma}{kimble}\label{lem:kimble}
    Let $P$ be a poset with ground set $X$. Let $P'$ be the poset with ground set $X \times \{0,1\}$ such that, for all $x,y\in X$, we have
    \begin{compactitem}
        \item[a)] $(x, 0) \not \prec_{P'} (y,0)$,
        \item[b)] $(x, 1) \not \prec_{P'} (y,1)$,
        \item[c)] $(y, 1) \not \prec_{P'} (x,0)$, and
        \item[d)] $(x, 0) \prec_{P'} (y,1) \Longleftrightarrow x \preceq_P y$.
    \end{compactitem}
    Then $\dim(P) \le \dim(P')$.
\end{restatable} 

Applying this lemma to a poset $P$ with ground set $X$, we obtain a poset $P'$ with ground set $X \times \{0,1\}$ such that all the edges in the comparability graph of $P'$ are between $X \times \{0\}$ and $X \times \{1\}$. This comparability graph is indeed bipartite with vertex classes $T=X \times \{1\}$ and $B=X \times \{0\}$, and every edge in this comparability graph is between some $(y,1)\in T$ and $(x,0)\in B$ such that $(x, 0) \prec_{P'} (y,1)$. 

Thus, in order to prove Theorem \ref{thm:main}, it suffices to prove the following version of the theorem with the additional assumption that the comparability graph of $P$ has such a special bipartite structure.

\begin{theorem}\label{thm:bipartite_dimension}
Let $\Delta\ge 6$ be an integer. Let $P$ be a poset whose comparability graph $G$ has maximum degree at most $\Delta$. Suppose that $G$ is bipartite with vertex classes $T$ and $B$ such that, for every edge $vw$ of $G$ with $v\in T$ and $w\in B$, we have $w\prec_P v$. Then $\dim(P)\le 21\Delta \log \Delta$.
\end{theorem}

\begin{proof}[Proof of Theorem \ref{thm:main} assuming Theorem \ref{thm:bipartite_dimension}]
    Let $P$ be a poset whose comparability graph has maximum degree at most $d$. We form the poset $P'$ as in Lemma \ref{lem:kimble}. By that lemma, then, we have $\dim(P) \le \dim(P')$. Furthermore, letting $\Delta=d+1$, the comparability graph $G$ of the poset $P'$ has degree at most~$\Delta$. Indeed, for each element $(x,0)\in X \times \{0\}$, the only neighbours of $(x,0)$ in $G$ are the (at most~$d$) elements $(y,1)\in X \times \{1\}$ with $x \prec_P y$ and the element $(x,1)$. Similarly, for each element $(y,1)\in X \times \{1\}$, the only neighbours of $(y,1)$ in $G$ are the (at most $d$) elements $(x,0)\in X \times \{0\}$ with $x \prec_P y$ and the element $(y,0)$.

    Observing that the poset $P'$ satisfies the assumptions in Theorem \ref{thm:bipartite_dimension} with $B=X \times \{0\}$ and $T=X \times \{1\}$, we can conclude from this theorem that
    \[\dim(P) \le\dim(P')\le 21\Delta \log \Delta=21(d+1) \log (d+1)\le 50 d\log d.\qedhere\]
\end{proof}

The remainder of this paper is devoted to proving Theorem~\ref{thm:bipartite_dimension}. Instead of directly constructing linear orders on $T\cup B$, we use a reformulation of F\"uredi and Kahn~\cite{furedi1986dimensions}. 
Any linear order on $T$ can be extended to a linear order on $T\cup B$ by adding the elements of $B$. In order for this to create a linear extension of $P$, each element of $B$ has to be before all of its neighbours in $T$ in the comparability graph $G$. An element $w\in B$ can be put directly before its smallest neighbour $v\in T$ in order to  `encode' as many `non-adjacencies' $w\nprec u$ with $w\in B$ and $u\in T$ as possible. (All the relationships $w\nprec u$ with $w,u\in B$, or $w,u\in T$, or $w\in T$ and $u\in B$, can be `encoded' easily with just two linear extensions of $P$.) 

This brings us to the goal of constructing $O(\Delta\log \Delta)$ linear orders on $T$ such that, for each $w\in B$ and each non-neighbour $u\in T$ of $w$, in one of these linear orders $u$ appears before all of the neighbours of $w$. 
The difficulty here is that we have to use very few linear orders. Letting $n=|T\cup B|$ be the ground set size, taking say $10\Delta \log n$ independent uniformly random linear orders of $T$ would, with high probability, satisfy this condition for each of the (at most $n^2$) relevant pairs $(w,u)$. Here, the term of $\log n$ is to enable a union bound over all of these pairs.

To circumvent this dependency on $n$, we will first colour the vertices in $T$ using $\Delta^{5}$ different colours. We then take a collection $\mathcal{L}$ of $O(\Delta\log \Delta)$ linear orders on the colours, and convert each $L\in \mathcal{L}$ into two linear orders on $T$: firstly listing the vertices of $T$ in order of their colours according to $L$ (ordering vertices of each colour arbitrarily) and then again listing the vertices of $T$ in order of their colours according to $L$ but where the order of 
the vertices of the same colour is reversed. The resulting linear orders on $T$ will have the property we need if, for every $w\in B$, the neighbours of $w$ have distinct colours and, for every non-neighbour $u\in T$ of $w$, there is an a linear order $L\in \mathcal{L}$ of the colours in which the colour of $u$ appears at least as early as the colour of any neighbour of $w$. Then, in one of the orders constructed on $T$, the vertex $u$ comes before all of the neighbours of $w$. (When $u$ has the same colour as some neighbour $v$ of $w$, the vertices $u$ and $v$ appear in opposite orders in the two linear orders on $T$ created from $L$. As there is at most one neighbour of $w$ of the same colour as $u$, in one of these two orders $u$ comes before all the neighbours of $w$.)

Now, the most natural approach would be to try to find a small collection $\mathcal{L}$ linear orders of the $\Delta^{5}$ colours, such that for any set $F$ of $\Delta$ colours, and any colour $c$, in one of the linear orders in $\mathcal{L}$, the colour $c$ appears before all colours in $F\setminus \{c\}$. Then we could satisfy the desired conditions by taking any colouring of the vertices of $T$ in which, for each $w\in B$, all neighbours of $w$ have distinct colours (e.g.\ by using a greedy colouring). This approach avoids the dependency on $n$, since we are now asking a problem stated purely in terms of $\Delta$ (namely, to find a collection $\mathcal{L}$ of $O(\Delta\log \Delta)$ linear orders of a set of size $\Delta^{5}$ with certain conditions). This is also essentially the approach of F\"uredi and Kahn~\cite{furedi1986dimensions} (see   the discussion after Lemma~\ref{lem:grouping_lemma}).

However, the key to our improved bound is to relax the conditions for the collection $\mathcal{L}$ of linear orders of the colours. Indeed, we only demand that for \emph{almost every} set $F$ of $\Delta$ colours it holds that for each colour $c$ there is a linear order in $\mathcal{L}$ in which the colour $c$ appears before all colours in $F\setminus \{c\}$.  While relaxing the condition for $F$, it is crucial to still take \emph{all} (and not almost all) colours $c$. 

Therefore, our collection $\mathcal{L}$ of $O(\Delta\log\Delta)$ linear orders of colours will not have the property we need for \emph{all} of the colourings of the vertices of $T$ in which, for each $w\in B$, all neighbours of $w$ have distinct colours, only for \emph{most} of the colourings. Nevertheless, choosing the colouring of the vertices randomly, we can show that the resulting colouring meets our requirements with positive probability, using the Lov\'asz Local Lemma in the following form. (For a proof of the Lov\'asz Local Lemma see, for example, \cite[Section~5.1]{AlonSpencer}.)

\begin{lemma}[Lov\'asz Local Lemma]
\label{lem:LLL}
Consider a sequence $X_1,\dots,X_n$ of independent random variables in some probability space. Let $I_1,\dots,I_m\su [n]$ be subsets of the index set $[n]$, and let $\mathcal{E}_1,\dots,\mathcal{E}_m$ be events where, for each $j\in [m]$, the event $\mathcal{E}_j$ only depends on the outcomes of the random variables $X_i$ with $i\in I_j$. Furthermore, let $p\ge 0$ be such that $\mathbb{P}[\mathcal{E}_j]\le p$ for all $j\in[m]$.
Finally, let $D>0$ be such that, for each $j\in [m]$, there are at most $D$ indices $j'\in [m]$ such that $I_j\cap I_{j'}\ne \emptyset$.

If $epD\leq 1$, then with positive probability none of the events $\mathcal{E}_1,\dots,\mathcal{E}_m$ holds.
\end{lemma}

\paragraph{Notation.}
All logarithms are to base $e$. For a positive integer $n$, we write $[n]=\{1,\dots,n\}$, as usual. For $0\le k\le n$, we furthermore write $\binom{[n]}{k}$ for the family of all $k$-element subsets of $[n]=\{1,\dots,n\}$. In a graph $G$, for a vertex $v$, we write $N(v)$ for the neighbourhood of $v$.

\section{Proof of Theorem \ref{thm:bipartite_dimension}}\label{sec:mainproof}

The following lemma is at the heart of our proof of Theorem \ref{thm:bipartite_dimension}.

\begin{lemma}
\label{lem:grouping_lemma} For each integer $\Delta\geq 6$, there exists a family $\mathcal{F}\subseteq \binom{[\Delta^{5}]}{\Delta}$
with 
\[
|\mathcal{F}|\geq \left(1-\frac{1}{\Delta^3}\right)\binom{\Delta^{5}}{\Delta}
\]
and a collection $\mathcal{L}$ of  linear orders of $[\Delta^{5}]$ of size $|\mathcal{L}|\le 10\Delta\log \Delta$ such that the following holds. For each $F\in \mathcal{F}$ and $x\in [\Delta^{5}]$, there is some $ L \in \mathcal{L}$ with $x \prec_L y$ for all $y\in F \setminus \{x\}$. 
\end{lemma}

Lemma \ref{lem:grouping_lemma} is arguably the main difference between our work and the work of F\"uredi and Kahn. A corollary of their main technical result, \cite[Lemma 3.3]{furedi1986dimensions}, is that if $\mathcal{F}$ is the family of all $\Delta$-subsets, then one can find a collection $\mathcal{L}$  of size $|\mathcal{L}| = O(\Delta^2 \log \Delta)$ satisfying the condition in Lemma \ref{lem:grouping_lemma}. By relaxing the condition to hold for every set $F\in \binom{[\Delta^{5}]}{\Delta}$ to only \emph{almost every} set, we save a factor of $\Delta$. Importantly, this relaxation still allows the main proof strategy to go through. We remark that they use that result with different parameters than above (leading to a loss of only $\log \Delta$ rather than $\Delta$), however, this illustrates where our improved bound comes from.

\begin{proof}[Proof of Lemma \ref{lem:grouping_lemma}] Let $t=\lfloor 10\Delta\log \Delta\rfloor\geq 8(\Delta+1)\log \Delta$. 
Let $\mathcal{L}$ be a set of $t$ linear orders on $[\Delta^{5}]$, each selected independently and uniformly at random from all such linear orders.
Let $\mathcal{F}$ be the family of $F\in \binom{[\Delta^{5}]}{\Delta}$ for which, for each $x\in [\Delta^{5}]$, there is some $L\in \mathcal{L}$ with $x \prec_L y$ for all $y\in F \setminus \{x\}$.
By the probabilistic method, it is then sufficient to show that $\mathbb{E}[|\mathcal{F}|]\geq \left(1-\frac{1}{\Delta^3}\right)\binom{\Delta^{5}}{\Delta}$.

Now, for each $F\in \binom{[\Delta^{5}]}{\Delta}$ and $x\in [\Delta^{5}]$, observe that, if $L$ is a uniformly random linear order of $[\Delta^{5}]$, then, by symmetry in the induced ordering of $F\cup \{x\}$, we have
\[
\mathbb{P}(x \prec_L y\text{ for all }y\in F \setminus \{x\})=\frac{1}{|F\cup\{x\}|}\geq \frac{1}{\Delta+1}.
\]
Therefore, for each  $F\in \binom{[\Delta^{5}]}{\Delta}$ and $x\in [\Delta^{5}]$, the probability there is no $L\in \mathcal{L}$ with $x \prec_L y$ for all $y\in F \setminus \{x\}$ is at most
\begin{equation}\label{eq:deltathirteen}
\left(1-\frac{1}{\Delta+1}\right)^{t}\leq \exp\left(-\frac{t}{\Delta+1}\right)\leq \exp(-8\log \Delta)= \frac{1}{\Delta^{8}}.
\end{equation}
Thus, for each $F\in \binom{[\Delta^{5}]}{\Delta}$, we have
\[
\mathbb{P}[F\in \mathcal{F}]\geq 1-\sum_{x\in [\Delta^{5}]}\mathbb{P}[\nexists L\in \mathcal{L}\text{ with\ }x \prec_L y\text{ for all }y\in F \setminus \{x\}]\overset{\eqref{eq:deltathirteen}}{\geq}1-\frac{\Delta^{5}}{\Delta^{8}}= 1-\frac{1}{\Delta^3}.
\]
By the linearity of expectation, we have $\mathbb{E}[|\mathcal{F}|]\geq \left(1-\frac{1}{\Delta^3}\right)\binom{\Delta^{5}}{\Delta}$, as required.
\end{proof}

Finally, we are ready for the proof of Theorem \ref{thm:bipartite_dimension}.

\begin{proof}[Proof of Theorem \ref{thm:bipartite_dimension}]
We apply Lemma~\ref{lem:grouping_lemma} to obtain a family $\mathcal{F}\subseteq \binom{[\Delta^{5}]}{\Delta}$ and a collection $\mathcal{L}$ of linear orders on $[\Delta^{5}]$ such that $|\mathcal{L}|\leq 10\Delta\log\Delta$, 
\[
|\mathcal{F}|\geq\left(1-\frac1{\Delta^3}\right)\binom{[\Delta^{5}]}{\Delta},
\]
and, for each $F\in \mathcal{F}$ and $x\in [\Delta^{5}]$, there is some $L\in \mathcal{L}$ with $x\prec_{L}y$ for all $y\in F\setminus \{x\}$.

Now, let $P$ be a poset with bipartite comparability graph $G$ with vertex classes $T$ and $B$ such that, for every edge $vw$ of $G$ with $v\in T$ and $w\in B$, we have $w\prec v$ in the poset $P$, and such that all the vertices in $G$ have degree at most $\Delta$.

\begin{claim}\label{claim:existence-colouring}
There exists a colouring $\phi:T\to [\Delta^{5}]$ of the vertices in $T$ with colours $1,\dots,\Delta^{5}$ such that the following two conditions hold for all $w\in B$:
\begin{compactitem}
        \item[(i)] $\phi(N(w))\subseteq F$ for some $F\in \mathcal{F}$ and
        \item[(ii)] $\phi(v)\neq\phi(v')$ for all distinct $v,v'\in N(w)$.
 \end{compactitem}
\end{claim}

\claimproofstart
Let $\phi:T\to [\Delta^{5}]$ be a random colouring of $T$ obtained by, for each $v\in T$, choosing the colour $\phi(v)\in [\Delta^{5}]$ independently and uniformly at random. For each $w\in B$, we define $\mathcal{E}_w$ as the event that at least one of the two conditions (i) and (ii) fails.
We want to show via the Lov\'asz Local Lemma (as stated as Lemma~\ref{lem:LLL}) that, with positive probability, none of the events $\mathcal{E}_w$ for $w\in B$ holds (i.e., that with positive probability the conditions (i) and (ii) hold for all $w\in B$).

Recall that the colours $\phi(v)$ for all $v\in T$ are independent random variables, and note that, for each $w\in B$, the event $\mathcal{E}_w$ only depends on the outcomes of $\phi(v)$ for $v\in N(w)$. Furthermore, for each vertex $w\in B$, there are at most $\Delta^2$ vertices $w'\in B$ such that $N(w)\cap N(w')\ne \emptyset$. Indeed, in order to have $N(w)\cap N(w')\ne \emptyset$, the vertices $w$ and $w'$ need to have a common neighbour $v\in T$. Given $w\in B$, there are at most $\Delta$ neighbours $v$ of $w$, and, for each such $v$, there are at most $\Delta$ choices for a neighbour $w'$ of $v$, so in total there are indeed at most $\Delta^2$ choices for $w'$. This means that, when applying  Lemma~\ref{lem:LLL} to the random variables $\phi(v)$ for $v\in N(w)$ and the events $\mathcal{E}_w$ for $w\in B$, we can take $D=\Delta^2$.

We now claim that for each $w\in B$ we have $\mathbb{P}[\mathcal{E}_w]\le 2/\Delta^3$. Let $w\in B$, and recall that $|N(w)|\leq  \Delta$. Thus, we always have $|\phi(N(w))|\le \Delta$, and so we can extend $\phi(N(w))\su [\Delta^{5}]$ to a set $F\subseteq [\Delta^{5}]$ of size $|F|=\Delta$ with $\phi(N(w))\su F$ by adding $\Delta-|\phi(N(w))|$ elements to $\phi(N(w))$ uniformly at random. 
    The resulting set $F$ is a uniformly random subset of $[\Delta^{5}]$ of size $ \Delta$. 
    Since $\mathcal{F}$ contains at least a $(1-1/\Delta^3)$-fraction of the subsets of $[\Delta^{5}]$ of size $ \Delta$, we find that $\mathbb{P}[F\not \in \mathcal{F}]\le 1/\Delta^3$. Note that, whenever we have $F\in \mathcal{F}$, condition (i) holds for $w$. Consequently, we obtain
    \[
    \mathbb{P}[\text{(i) fails for }w]\leq \mathbb{P}[F\not \in \mathcal{F}]\leq \frac{1}{\Delta^3}.
    \]
    Furthermore, we observe that
       \[
    \mathbb{P}[\text{(ii) fails for }w]\leq \binom{\Delta}2 \frac1{\Delta^{5}}\leq \frac1{\Delta^3}.
    \]
    Indeed, for each of the $\binom{\Delta}2$ choices of distinct vertices $v,v'\in N(w)$, the probability that $\phi(v)=\phi(v')$ is $1/\Delta^{5}$. All in all, we obtain
    \[
    \mathbb{P}(\mathcal{E}_w)\leq \mathbb{P}[\text{(i) fails for }w]+\mathbb{P}[\text{(ii) fails for }w]\leq \frac{1}{\Delta^3}+\frac1{\Delta^3}=\frac{2}{\Delta^3}.
    \]
    We can therefore apply Lemma \ref{lem:LLL} (the Lov\'asz Local Lemma) with $D=\Delta^2$ and $p=2/\Delta^3$ to the events $\mathcal{E}_w$ for $w\in B$, noting that $epD=2e/\Delta\le 1$ (since $\Delta\geq 6$).
    Now, by Lemma \ref{lem:LLL}, with positive probability, none of the events $\mathcal{E}_w$ for $w\in B$ hold (meaning that (i) and (ii) are satisfied for all $w\in B$). Thus, there exists a colouring $\phi:T\to [\Delta^{5}]$ satisfying (i) and (ii) for all $w\in B$.
\claimproofend

The properties of a colouring given by the previous claim are used to deduce the existence of a collection of linear orders on $T$ which, ultimately, will allow us to control non-adjacencies in the comparability graph $G$.

\begin{claim}\label{claim:L-star}
There exists a collection $\mathcal{L}^*$ of $2|\mathcal{L}|$ linear orders on $T$ such that, for all $w\in B$ and $u\in T\setminus N(w)$, there exists $L^*\in \mathcal{L}^*$ with $u\prec_{L^*} v$ for all $v\in N(w)$. 
\end{claim}
\claimproofstart
Fix a colouring $\phi:T\to [\Delta^{5}]$ as in Claim \ref{claim:existence-colouring}. This colouring gives rise to a partition of $T$ into the colour classes $\phi^{-1}(i)$ for $i\in [\Delta^{5}]$. For each $i\in [\Delta^{5}]$, we fix an arbitrary linear order $R_i$ of the colour class $\phi^{-1}(i)$. 

For each $L \in \mathcal{L}$,  we now form two linear orders $L^+$ and $L^-$ on $T$ as follows.
Informally, $L$ determines the order of the colour classes $\phi^{-1}(i)$; within each colour class $\phi^{-1}(i)$, $L^+$ follows the order $R_i$ and $L^-$ follows the reverse of the order $R_i$. More formally, we define the linear orders $L^+$ and $L^-$ on $T$ such that, for distinct $v,v'\in T$ and $i\in [\Delta^{5}]$, the following holds:
\begin{itemize}
    \item  if $\phi(v)=\phi(v')=i$, then
\[
v\prec_{L^+}v' \iff v\prec_{R_i}v' \iff
v'\prec_{L^-}v,
\]
\item and if $\phi(v)\neq \phi(v')$, then
\[
v\prec_{L^+}v' \iff \phi(v)\prec_L\phi(v') \iff
v\prec_{L^-}v'.
\]
\end{itemize}
Let $\mathcal{L}^*$ denote the collection of the linear orders $L^+$ and $L^-$ obtained in this way for all $L\in \mathcal{L}$.

We now show that, for each $w\in B$ and $u\in T\setminus N(w)$, there is a linear order in $\mathcal{L}^*$ for which $u$ appears before all vertices in $N(w)$ (that is, $u\prec v$ for all $v\in N(w)$).
By condition~(i) for the colouring $\phi$, there exists $F \in \mathcal{F}$ with $\phi(N(w))\subseteq F$. By the property of $\mathcal{F}$ (applied with $x=\phi(u)$), there exists $L\in \mathcal{L}$ with $\phi(u)\prec_L y$ for all $y\in F\setminus \{\phi(u)\}$.
This implies that $\phi(u)\prec_{L} \phi(v)$ for each $v\in N(w)$ such that $\phi(v)\neq \phi(u)$.
By definition, for any such $v\in N(w)$, we find that $u\prec_{L^+}v$ and $u\prec_{L^-}v$.

By condition (ii) for the colouring $\phi$, there is at most one vertex $v\in N(w)$ with $\phi(v)=\phi(u)$. If such a vertex $v\in N(w)$ does exist, then $u$ appears before $v$ in exactly one of
the linear orders $\prec_{L^+}$ or $\prec_{L^-}$ (this depends on the ordering $R_{\phi(u)}$ within the colour class of $u$ and $v$), and for that linear order $u$ appears before all vertices in $N(w)$. Therefore, for all $w\in B$ and $u\in T\setminus N(w)$ there exists $L\in \mathcal{L}$ such that $u$ appears before $N(w)$ in either $\prec_{L^+}$ or $\prec_{L^-}$ (or both).
\claimproofend

Let $\mathcal{L}^*$ be a collection of linear orders on $T$ as in Claim \ref{claim:L-star}.
We use this to construct a collection $\mathcal{L}'$ of linear orders on $T\cup B$ whose intersection is the poset $P$. 
We first fix arbitrary orders $R_B$ on $B$ and $R_T$ on $T$ and put the following two linear orders into $\mathcal{L}'$:
\begin{itemize}
    \item First take all vertices in $B$ in the order $R_B$ and then take all vertices in $T$ in the order $R_T$.
    \item First take all vertices in $B$ in the reverse of the order $R_B$ and then take all vertices in $T$ in the reverse of the order $R_T$.
\end{itemize}
Both of these linear orders are linear extensions of $P$. Furthermore, together they ensure that, for any two elements $u,w\in T\cup B$ with $u,w\in B$, or with $u,w\in T$, or with $u\in B$ and $w\in T$, there is a linear order $L'\in \mathcal{L}'$ with $u\prec_{L'} w$.

For all $L^*\in \mathcal{L}^*$, we add the following linear order $L'$ on $T\cup B$ to $\mathcal{L}'$. We define $L'$ by ordering the vertices in $T$ according to $L^*$, and then iteratively inserting the vertices $w \in B$ one by one (say, according to the order in $R_B$) as follows:
\begin{itemize}
    \item if $N(w)\neq\emptyset$, then let $v \in N(w)$ be the smallest element in $N(w)$ according to the linear order $L^*$, and add $w$ directly before $v$ in $L'$;
    \item otherwise, if $N(w)=\emptyset$, then add $w$ at the end of the current linear order $L'$, after all the other vertices in $L'$.
\end{itemize}
By construction, $L'$ is a linear extension of $P$, because every vertex $w\in B$ is inserted so that it is before all of its neighbours.

Let us now check that, for every $w\in B$ and $u\in T\setminus N(w)$, there is some linear order in $\mathcal{L}'$ in which $u$ appears before $w$. Indeed, by  the condition in Claim \ref{claim:L-star}, there is $L^*\in \mathcal{L}^*$ with $u\prec_{L^*} v$ for all $v\in N(w)$. When constructing $L'$ from $L^*$ as above, the vertex $w$ is inserted directly before the smallest vertex $v\in N(w)$ according to the linear order $L^*$, meaning that it is inserted into a position after $u$. Thus, we also have $u\prec_{L'} w$ in the final linear order $L'\in \mathcal{L}'$.

All in all this shows that $\mathcal{L}'$ is a collection of linear extensions of $P$ such that for all $u,w\in T\cup B$ with $w\not\prec_P u$, there is a linear order $L'\in \mathcal{L}'$ with $u\prec_{L'} w$. 
We conclude that
\[
\dim(P)\le |\mathcal{L}'|\le |\mathcal{L}^*|+2= 2|\mathcal{L}|+2\leq 20\Delta\log\Delta+2\le 21\Delta\log\Delta.\qedhere
\]
\end{proof}

\paragraph{AI statement.} AI models were only used for searching the literature and generating the TikZ code for Figure~\ref{fig}. The proofs and writing are entirely due to the authors.

\paragraph{Acknowledgements.} This work was carried out at Mathematisches Forschungsinstitut Oberwolfach
(MFO) during the ``MATRIX-MFO Tandem Workshop: Combinatorial Interchange'' in September 2026. We thank the organisers of the workshop and the MFO for the productive working environment, Yelena Yuditsky for helpful discussions, and Jacob Fox for pointing us to \cite{trotter1996graphs}.

\bibliographystyle{abbrv}
\bibliography{poset}

\begin{thebibliography}{1}

\bibitem{adiga2011boxicity}
A.~Adiga, D.~Bhowmick, and L.~S. Chandran.
\newblock Boxicity and poset dimension.
\newblock {\em SIAM Journal on Discrete Mathematics}, 25(4):1687--1698, 2011.

\bibitem{AlonSpencer}
N.~Alon and J.~H. Spencer.
\newblock {\em The Probabilistic Method}.
\newblock John Wiley \& Sons, 4th edition, 2016.

\bibitem{erdos1991dimension}
P.~Erd\H{o}s, H.~A. Kierstead, and W.~T. Trotter.
\newblock The dimension of random ordered sets.
\newblock {\em Random Structures \& Algorithms}, 2(3):253--275, 1991.

\bibitem{furedi1986dimensions}
Z.~F{\"u}redi and J.~Kahn.
\newblock On the dimensions of ordered sets of bounded degree.
\newblock {\em Order}, 3(1):15--20, 1986.

\bibitem{scott2020better}
A.~Scott and D.~Wood.
\newblock Better bounds for poset dimension and boxicity.
\newblock {\em Transactions of the American Mathematical Society}, 373(3):2157--2172, 2020.

\bibitem{trotter1996graphs}
W.~T. Trotter.
\newblock Graphs and partially ordered sets: recent results and new directions.
\newblock {\em Congressus Numerantium}, pages 253--278, 1996.

\bibitem{TrotterMoore1976}
W.~T. Trotter and J.~I. Moore.
\newblock Characterization problems for graphs, partially ordered sets, lattices, and families of sets.
\newblock {\em Discrete Mathematics}, 16(4):361--381, 1976.

\end{thebibliography}

\appendix

\section{Proof of Lemma~\ref{lem:kimble}}
Lemma~\ref{lem:kimble} was originally proved by Kimble (see, e.g., \cite{trotter1996graphs}), and a proof can be found in Trotter's book on posets~\cite{trotter1996graphs}. For the reader's convenience, we include a simplified version of the proof here.
\kimble*
\begin{proof}[Proof of Lemma \ref{lem:kimble}] Let $d=\dim(P')$ and let $\{L_1', \ldots, L_d'\}$ be linear orders on $X \times \{0,1\}$ such that $P' = \bigcap_{i = 1}^d L'_i$. For each $i \in [d]$, form a linear order $L_i$ on $X$ as follows: 
    \begin{itemize}
        \item Initially, let $L_i^0$ be the linear order on $X$ given by the restriction of $L_i'$ to $X \times \{1\}$.
        \item For $j = 1, \ldots, n$, sequentially, define $L_i^j$ as follows: Let $(z_1, \ldots, z_n)$ be the permutation of $X$ corresponding to $L_i^{j-1}$. If there is no $j' \in \{1, \ldots, j-1\}$ with $z_j \prec_P z_{j'}$, then set $L_i^j := L_i^{j-1}$. Otherwise, let $j_0 \in \{1, \ldots, j - 1\}$ be the smallest index such that $z_j \prec_P z_{j_0}$, and form $L_i^j$ from $L_i^{j-1}$ by moving $z_j$ to between $z_{j_0 - 1}$ and $z_{j_0}$ (where if $j_0 = 1$, then we move $z_j$ to the beginning). More informally, we move $z_j$ as little as possible (and perhaps not at all) to earlier in the sequence so that it comes before any $z_{j'}$ for which $z_j \prec_P z_{j'}$.

        \item Let $L_i= L_i^n$.
    \end{itemize}
    
We claim that $P = \bigcap_{i = 1}^d L_i$. For this, we need to show that, for every $x,y\in X$ with $x \prec_P y$ we have $x\prec_{L_i} y$ for \emph{each} $i\in [d]$ and, for every $x,y\in X$ with $x \not\prec_P y$, we have $y\prec_{L_i} x$ for \emph{some} $i\in [d]$. 

Suppose, for contradiction, that there are some $x,y\in X$ and $i\in [d]$ with $x \prec_P y$ but $y \prec_{L_i} x$. Take such a pair $x,y$ so that the latest of $x$ and $y$ in $L_i^0$ comes as early as possible. If $y \prec_{L_i^0} x$, then, in defining $L_i$, $x$ was processed after $y$ and thus moved somewhere earlier than $y$ as $x \prec_P y$. After $x$ and $y$ are processed, their relative order does not change in subsequent rounds in the construction of $L_i$, and hence $x \prec_{L_i} y$, a contradiction. Thus, we may assume that $x \prec_{L_i^0} y$, and so $y$ was processed after $x$. Now, when processing $y$ (say at step $j$, where $L_i^{j}$ is obtained from $L_i^{j-1}$ by moving $y$), $y$ must be moved to a position before $x$. This can only happen if there is some $z$ earlier than $x$ in $L_i^{j-1}$ (which then implies $z\prec_{L_i}x$ and $z\prec_{L_i^0} y$) with $y\prec_P z$. Then, by transitivity, we have $x\prec_P z$. In particular, $x$ and $z$ are also a violating pair but $z$ comes earlier than $y$ in $L_i^0$. This contradicts the choice of $x$ and $y$. Therefore, there are no $x,y\in X$ and $i\in[d]$ with $x\prec_P y$ and $y \prec_{L_i} x$.

 Next, suppose for contradiction that there exist distinct $x,y\in X$ with $x\not\prec_P y$ such that for all $i\in [d]$ we have $x\prec_{L_i}y$. 
 Let 
    $$
        X' = \{x' \in X \; : \; x \preceq_P x'\}.
    $$
As $(x,0)\not\prec_{P'}(y,1)$,  we can find some $i\in [d]$ with $(y, 1) \prec_{L'_i} (x, 0)$.
    For each $x'\in X$, we have $x \preceq_P x'$, hence $(x,0)\prec_{P'}(x',1)$, so $(x,0)\prec_{L_i'}(x',1)$, and therefore $(y, 1) \prec_{L_i'}(x',1)$. Thus, every $x'\in X'$ comes after $y$ in $L_i^0$, so at the beginning of the construction process for $L_i$, all elements of $X'$ are to the right of $y$. Since $x\prec_{L_i}y$, at the end of the process the element $x\in X'$ is on the left of $y$ in $L_i=L_i^n$. Thus, we can consider the first moment in the process (i.e., the first $L_i^j$) where an element of $X'$ appears to the left of $y$, and call this element $x'\in X'$. This moment must be when the element $x'$ is processed and, to cause it to move to the left of $y$, there must have already been some $z\in X$ somewhere to the left of $y$ such that $x'\prec_P z$. But, by transitivity, we then have $x\prec_P z$ and hence $z\in X'$. This is a contradiction, as we considered the first moment in the process where there is an element of $X'$ to the left of $y$. Therefore, for every distinct $x,y\in X$ with $x\not\prec_P y$ there is some $i\in [d]$ with $y\prec_{L_i}x$.
\end{proof}
\end{document}